\documentclass[a4paper,11pt]{article}

\usepackage{amsmath,amssymb,amsthm}

\usepackage{accents}

\usepackage{geometry}
\usepackage{bbm}
\usepackage{multirow}
\usepackage{pgfplots}
\usepackage{mathtools}
\usepackage{todonotes}
\usepackage{comment}
\usepackage[ruled,vlined,linesnumbered]{algorithm2e}

\allowdisplaybreaks[1]

\newcommand{\set}[2]{\left\{#1 \; \left|\; #2 \right.\right\}}
\newcommand{\indicator}{\mathbbm{1}}

\newcommand{\UG}{\Delta^\Gamma}

\newcommand{\Y}{\mathcal{Y}}

\newcommand{\I}{{\mathcal{I}}}
\newcommand{\bI}{\bar{I}}

\newcommand{\W}{\mathcal{W}}
\newcommand{\X}{\mathcal{X}}

\newcommand{\one}{\mathbf{1}}

\newcommand{\GammaI}{{\hat{\Gamma}}}
\newcommand{\setGammaI}{{\mathcal{G}_I}}

\newcommand{\Z}{\mathbb Z}

\newcommand{\Gammaone}{\GammaI}
\newcommand{\Gammaonestar}{\Gamma_{I^*}}
\newcommand{\Gammazero}{\bar{\Gamma}}
\newcommand{\Gammazerostar}{\Gammazero^*}
\newcommand{\pospart}[1]{\left[#1\right]^+}

\newcommand{\bx}{\pmb{x}}
\newcommand{\bbf}{\pmb{f}}
\newcommand{\bg}{\pmb{g}}
\newcommand{\bh}{\pmb{h}}

\newcommand{\bbdelta}{\pmb{\delta}}
\newcommand{\bgamma}{\pmb{\gamma}}

\DeclareMathOperator*{\argmax}{arg\,max}
\DeclareMathOperator*{\argmin}{arg\,min}
\DeclareMathOperator*{\conv}{conv}

\renewcommand{\mp}[1]{{\color{purple}#1}}

\newcommand{\MP}[1]{~\\[4pt]\todo[inline,color=green!40,noinlinepar,size=\small]{\textbf{MP:} #1}}

\newtheorem{observation}{Observation}
\newtheorem{corollary}{Corollary}
\newtheorem{proposition}{Proposition}
\newtheorem{theorem}{Theorem}

\newtheorem{lemma}{Lemma}
\newtheorem{assumption}{Assumption}

\usepackage{authblk}

\title{The robust selection problem with information discovery} 

\author[1]{Xiaoyu Chen\footnote{Corresponding author. Email: xiaoyu.chen@lirmm.fr}}
\author[2]{Marc Goerigk}
\author[1]{Michael Poss}

\affil[1]{LIRMM, University of Montpellier, CNRS, Montpellier, France}
\affil[2]{Business Decisions and Data Science, University of Passau, Germany}

\date{}

\begin{document}

\maketitle

\begin{abstract}
We explore a multiple-stage variant of the min-max robust selection problem with budgeted uncertainty that includes queries. First, one queries a subset of items and gets the exact values of their uncertain parameters. Given this information, one can then choose the set of items to be selected, still facing uncertainty on the unobserved parameters. In this paper, we study two specific variants of this problem. The first variant considers objective uncertainty and focuses on selecting a single item. The second variant considers constraint uncertainty instead, which means that some selected items may fail. We show that both problems are NP-hard in general. We also propose polynomial-time algorithms for special cases of the sets of items that can be queried. For the problem with constraint uncertainty, we also show how the objective function can be expressed as a linear program, leading to a mixed-integer linear programming reformulation for the general case. We illustrate the performance of this formulation using numerical experiments.\\

\end{abstract}


\section{Introduction}

Robust optimization is an increasingly popular approach to handle the uncertainty that arises in mixed-integer linear optimization problems. Leading to compact reformulations and requiring only to describe the uncertain parameters by their possible realizations, robust optimization is easy to apply and to manipulate. Among the different uncertainty sets that have been proposed in the literature, budgeted uncertainty is particularly convenient for combinatorial problems as it 
essentially preserves the complexity of the nominal problem when the number of robust constraints is constant~\cite{BertsimasS03,Alvarez-MirandaLT13}. The model considers that each uncertain cost (resp. weight) has a nominal cost $\bar{c}_j$ (resp. nominal weight $\bar a_{ij}$) and a possible deviation $\hat{c}_j$ (resp. $\hat a_{ij})$ and that in any possible scenario, at most $\Gamma\in\Z$ uncertain parameters will take their high costs $\bar{c}_j+\hat{c}_j$ (resp. high weight $\bar{a}_{ij}+\hat{a}_{ij}$), the other ones taking instead the low value $\bar{c}_j$ (resp. $\bar a_{ij}$). This uncertainty model is motivated by the observation that it is \emph{unlikely} that all components of the cost or weight vector simultaneously take their worst-case, which can be formalized with the probabilistic bounds discussed in~\cite{BertsimasS04,Poss13}. Defining $[n]=\{1,\ldots,n\}$ and
$$
\UG=\set{\pmb{\delta} \in \{0,1\}^n}{\sum_{i\in [n]}\delta_i \leq \Gamma},
$$
a general robust combinatorial optimization problem with $n$ variables and $m$ robust constraints can be defined as
\begin{equation}
\label{eq:rcoGamma}\tag{RCO}
 \min_{\bx\in \X} \set{\max_{\pmb{\delta} \in \UG} \sum_{j\in[n]}(\bar{c}_j+\delta_j\hat{c}_j)x_j}{\sum_{j\in[n]}(\bar{a}_{ij}+\delta_j\hat{a}_{ij})x_j\leq b_i,\;\forall i\in[m],\forall \bbdelta\in\Delta^\Gamma},
\end{equation}
where $\X\subseteq\{0,1\}^n$ is the feasibility set of the deterministic problem. In addition to its theoretical tractability, and its probabilistic bounds, strong formulations~\cite{BusingGK23IPCO} and ad-hoc branch-and-bound algorithms~\cite{BusingGK23} have recently been proposed, leading to efficient algorithms for solving problem~\eqref{eq:rcoGamma}.

This optimistic outlook has made possible to consider more complex robust problems than before. Among these richer models, some now consider the \emph{endogenous nature} of the uncertain parameters, meaning that the value taken by the uncertain parameters depend on the decision that has been taken. Decision-dependent uncertainty arises in a wide variety of settings, in particular when some investment can be made to reduce the size of the uncertainty set~\cite{ArslanP24,nohadani2018optimization}. Herein we are interested in the problem where some binary measurement (or queries) model the investments that can be made on the uncertain parameters, leading to the following problem with four levels of decisions: (i) the set of queried parameters $I$ is selected, (ii) the adversary reveals the values of the queried parameters with vector $\bgamma$, (iii) the actual solution $\bx\in \X$
is selected, and (iv) the adversary chooses the value for the remaining uncertain parameters in $\UG(I,\pmb{\gamma})=\set{\pmb{\delta}\in\UG}{\delta_i=\gamma_i,\,\forall i \in I}$, which enforces the coherence between the uncertain vector $\bbdelta\in\UG(I,\pmb{\gamma})$ and the initial observations $\bgamma\in\UG$. The complete decision-dependent information discovery (DDID) problem can thus be stated as follows:
\begin{equation}
\label{eq:DDID}\tag{DDID}
    \min_{I\in \I}\max_{\pmb{\gamma} \in \UG} \min_{\bx\in \X} \set{\max_{\pmb{\delta} \in \UG(I,\pmb{\gamma})} \sum_{j\in[n]}(\bar{c}_j+\delta_j\hat{c}_j)x_j}{\sum_{j\in[n]}(\bar{a}_{ij}+\delta_j\hat{a}_{ij})x_j\leq b_i,\;\forall i\in[m],\forall \bbdelta\in\UG(I,\pmb{\gamma})},
\end{equation}
where $\I\subseteq 2^{[n]}$ contains the sets of items that can be queried simultaneously. Observe that the values of $\bgamma$ chosen for $j\in[n]\setminus I$ in the outermost maximization are irrelevant since these can be freely modified when choosing the vectors $\bbdelta\in\UG(I,\pmb{\gamma})$.

Model~\eqref{eq:DDID} has many applications in urban planning, project management, resource allocation, and scheduling, among many others. Early papers involving DDID where, for instance, considering applications production planning~\cite{jonsbraaten1998class}. The more recent paper~\cite{vayanos2021robust} details applications in a R\&D project portfolio optimization problem, where a company must choose how to prioritize the projects in its pipeline~\cite{solak2010optimization,colvin2008stochastic}.
That paper also describes a preference elicitation with real-valued recommendations where one can investigate how much users like any particular item. 
The paper further applies the latter model to improve the US kidney allocation system.
Even more recently, Paradiso et al.~\cite{paradisoexact} consider combining a routing problem with sensor location, which is motivated by the problem of collecting medicine crates at the Alrijne hospital.

Despite the many applications of~\eqref{eq:DDID}, solution methodology for the problem has stayed behind. To the best of our knowledge, the only general approach to~\eqref{eq:DDID} is that of~\cite{vayanos2021robust}, who proposed a heuristic reformulation based on the application of the $K$-adaptability paradigm~\cite{hanasusanto2015k}. Unfortunately, the resulting MILP reformulations are large and heavily rely on big-$M$ coefficients, therefore limiting their use to rather small instances. The problem appears to be easier to solve when there are no robust constraints ($m=0$), and~\cite{paradisoexact} and~\cite{omer2023combinatorial} proposed exact algorithms in this case. On the one hand,~\cite{paradisoexact} proposed an exact combinatorial Benders decomposition algorithm which can reasonably well solve~\eqref{eq:DDID}, making no particular assumption about $\X$. On the other hand,~\cite{omer2023combinatorial} leverages the explicit description of $\conv(\X)$ to propose strong MILP reformulations, possibly coupled with a branch-and-price algorithm in the case the description of $\conv(\X)$ is obtained through Dantzig-Wolfe decomposition.

Furthermore, very little is known about the theoretical complexity even for
very simple cases of~\eqref{eq:DDID}, which is the gap we intend to start filling in this work. Specifically, we consider two simple cases of~\eqref{eq:DDID} based on the selection problem, which is characterized by the feasibility set $\X^{SEL}=\set{\bx\in\{0,1\}^n}{\sum_{j\in[n]}x_i=p}$. Due to their simple structure, selection problems are frequently studied in robust combinatorial optimization, see, e.g., \cite{conde2004improved,kasperski2017robust,chassein2018recoverable,lachmann2021linear}.
We show under which conditions the resulting special cases~\eqref{eq:DDID} are polynomially solvable or NP-hard. First, let us recall that solving the (deterministic) selection problem for cost vector $\pmb{c}$ is equivalent to  optimizing $\pmb{c}^\top \bx$ over $\X^{SEL}$, which can be done in $O(n)$ thanks to the following result, see also~\cite{robook} for an overview.
\begin{theorem}[\cite{blum1973time}]\label{thm:On}
The $i$-th smallest of $n$ numbers can be computed in $O(n)$.
\end{theorem}
\noindent By determining the $p$-th smallest item with respect to $\pmb{c}$, we can then construct an optimal solution to the selection problem in linear time.

The first of the two problems we study in this paper considers more specifically the 1-selection problem under \emph{Objective Uncertainty} (thus assuming $m=0$):
\begin{equation}
\label{eq:OU}
    \min_{I\in \I}\ \max_{\pmb{\gamma} \in \UG}\ \min_{j\in[n]}\ \max_{\pmb{\delta} \in \UG(I,\pmb{\gamma})}\ \bar{c}_j+\delta_j\hat{c}_j,\tag{OU}
\end{equation}
where the inner minimization over $j$ has replaced the minimum over $\pmb{x}\in \X$ because a unique item is being chosen (i.e., $p=1$). We mention that problem~\eqref{eq:OU} has been first introduced in~\cite{vayanos2021robust} under the name of \emph{preference elicitation}. This problem refers to developing a decision support system capable of generating recommendations to a user, by letting that person ask a limited number of questions from a potentially large set before making a recommendation~\cite{vayanos2020robust}. In what follows, we denote the objective function of the first stage of the problem as
\begin{equation*}
\Psi(I)=\max_{\pmb{\gamma} \in \UG}\ \min_{j\in[n]}\ \max_{\pmb{\delta} \in \UG(I,\pmb{\gamma})}\ \bar{c}_j+\delta_j\hat{c}_j.
\end{equation*}

The second problem we study considers a variant of the selection problem where some selected items can fail, and it is uncertain which ones will fail. We may query items beforehand to determine whether they will fail. We want to ensure that amongst the $p$ selected items, at most $b$ will fail. Protecting against failing items is an often-studied variant of robust combinatorial optimization, see, e.g., \cite{adjiashvili2015bulk,goerigk2024robust}.
We obtain the \emph{Constraint Uncertainty} optimization problem from~\eqref{eq:DDID} by considering a single robust constraint with $\bar a_j=0$ and $\hat a_j=1$ and a known objective:
\begin{equation}
\label{eq:CU}
    \min_{I\in \I}\ \max_{\pmb{\gamma} \in \UG}\ \min_{\pmb{x}\in \X^{SEL}} \set{\sum_{i\in[n]} c_ix_i}{\sum_{i\in[n]} \delta_i x_i \leq b,\ \forall \pmb{\delta}\in \UG(I,\pmb{\gamma})}. \tag{CU}
\end{equation}
Observe that we allow an item to be packed, even if we know that it will fail.
In what follows, we denote the objective function of this problem as
\begin{equation*}
\Phi(I)=\max_{\pmb{\gamma} \in \UG}\ \min_{\pmb{x}\in \X^{SEL}} \set{\sum_{i\in[n]} c_ix_i}{\sum_{i\in[n]} \delta_i x_i \leq b,\ \forall \pmb{\delta}\in \UG(I,\pmb{\gamma})}.
\end{equation*}
While we use the terminology that items ``fail'' for simplicity, the meaning of the uncertain constraint will depend on the application that is modeled. For example, if we select a team of $p$ employees, the uncertain constraint means that at most $b$ should have some specific property; if we select a basket of $p$ options within a recommender system, at most $b$ should be of a specific category; or if we select $p$ assets for a project portfolio, at most $b$ should invoke a specific resource. More generally, problems of type~\eqref{eq:CU} model cardinality-constrained problems (see \cite{stephan2010cardinality}), where it is uncertain if an item contributes to the cardinality constraint.

\paragraph{Contributions and structure of the paper} We first consider in Section~\ref{sec:OU} the case of $1$-selection problem under objective uncertainty~\eqref{eq:OU}. We provide analytical solutions for $\Phi(I)$, and prove the problem to be NP-hard when $\I$ is arbitrary. We then focus on the set that allows to query any set of $q$ items
\begin{equation*}
\I^{SEL}=\set{I \subseteq[n]}{|I| \leq q},
\end{equation*}
and show that optimizing $\Phi(I)$ over this set can be done in $O(n)$ by querying the $q$ smallest items in terms of $\bar{\pmb{c}}$. We then extend our result to a generalization of that set, the knapsack set defined by a given vector of weights $\pmb{a}$ and a capacity $C$,
\begin{equation*}
\I^{KP}=\set{I \subseteq[n]}{\sum_{i\in I} a_i  \leq C}.
\end{equation*}
We obtain that, somewhat surprisingly, optimizing $\Phi(I)$ over $\I^{KP}$ can be done in $O(n\log n)$ by using a slightly more complicated algorithm. We then turn to the $p$-selection problem with constraint uncertainty~\eqref{eq:CU} in Section~\ref{sec:CU} and provide conditions on the problem's input to ensure its feasibility. We then show that the problem is NP-hard already for $\I^{KP}$. Then, we provide a linear programming reformulation for $\Phi(I)$. The formulation can be further leveraged to obtain a mixed-integer linear programming reformulation for the full problem and any set $\I$ for which we have a compact formulation. The construction of the linear programming reformulations combines the linear programming duality classically used in robust optimization with some new ad-hoc ingredients specific to~\eqref{eq:CU}. Last, we focus on the case of $\I^{SEL}$ and show which items should always be selected in an optimal $I$, leading again to an algorithm in $O(n)$ thanks to Theorem~\ref{thm:On}. 
The MILP reformulation is illustrated with numerical experiments in Section~\ref{sec:num}, and the paper is concluded in Section~\ref{sec:conc}.

\paragraph{Additional references} We mention that other models with queries have been considered in the literature, for instance under the name of explorable uncertainty, which is usually traced back to~\cite{kahan1991model}. In that work, the author studies the problem of how many queries are necessary to find the optimal solution to the selection problem with cost uncertainty. Subsequent papers have later studied the minimal number of queries to guarantee optimality for other optimization problems, such as the shortest path problem~\cite{feder2007computing}, the knapsack problem~\cite{goerigk2015robust}, and the spanning tree problem~\cite{erlebach2008computing,megow2017randomization}.

\paragraph{Further notation}
For any integer $n$, we denote by $[n]_{0}$ the set $\{0, 1,2,\dots,n\}$. For any real number $a$, we denote $\max \{0, a\}$ by $[a]^{+}$. Set a set $I\subseteq[n]$, we denote by $\bI=[n]\setminus I$ all items not in $I$. Throughout, we also use the notation $\Gammaone=\sum_{i \in I} \gamma_i$ and $\Gammazero=\Gamma-\Gammaone$ to distinguish the budgets spent on the observed and unobserved items, respectively. Furthermore, we denote by $\setGammaI=[\min(|I|,\Gamma)]_0$ the set of values that $\Gammaone$ can take. We also introduce the set
$
\Delta^{\Gamma}_= = \set{\pmb{\delta} \in \{0,1\}^n}{\sum_{i\in [n]}\delta_i = \Gamma}
$
 to represent the case where the budget constraint must be satisfied at equality. Last, as we focus on the selection feasibility set in what remains, we denote $\X^{SEL}$ shortly as $\X$.

\section{Objective uncertainty}
\label{sec:OU}

In this section, we discuss problem~\eqref{eq:OU}. We provide elementary observations in Section~\ref{subsec:Objective Additional definitions and a basic property}. Given a query set $I$, we derive a closed-form expression of the optimal objective value in Section~\ref{subsec: Expression of Psi(I)}. We prove in Section~\ref{subsec: NP-hardness of the general case} that the whole problem is NP-hard in general. In Section~\ref{subsec: Polynomial-time algorithms for special cases}, we show that the special cases that use query sets $\I^{SEL}$ or $\I^{KP}$ can be solved in polynomial time.

\subsection{Additional definitions and basic properties}
\label{subsec:Objective Additional definitions and a basic property}

To simplify expressions, for any $I\in \I$ and $\Gammaone\in \setGammaI$, we define the cost of query set $I$ using exactly $\Gammaone$ deviations on $I$ and $\Gammazero=\Gamma-\Gammaone$ deviations on $\bI$ as
\begin{equation*}
\psi(I,\Gammaone) = \max_{\pmb{\gamma} \in \Delta^{\Gammaone}_=}\ \min_{j\in [n]}\ \max_{\pmb{\delta}\in \Delta^{\Gamma-\Gammaone}_=(I,\pmb{\gamma})} \ \bar{c}_j+\delta_j\hat{c}_j.
\end{equation*}
Observe that $\Psi(I) = \max_{\Gammaone\in\setGammaI} \psi(I,\Gammaone)$.
Throughout this section, we assume the following to ease the analysis of the algorithms.
\begin{assumption}
\label{assump: item sorted by bar c}
    We assume that the items are sorted such that $\bar{\pmb{c}}$ is non-decreasing, and this order is maintained within the subsets $I$ and $\bI$.
\end{assumption}
Observe that to ensure Assumption~\ref{assump: item sorted by bar c}, a sorting step in $O(n \log n)$ is necessary. However, for algorithms that run in $O(n)$, we can justify that the assumption is for ease of presentation only.
We provide next two simple observations that will be used throughout.
\begin{observation}
\label{obs: lower_upper bound}
    The values $\min_{j\in[n]} \bar{c}_j$ and  $\min_{j\in[n]} \bar{c}_j+\hat{c}_j$ represent a lower bound and an upper bound of problem~\eqref{eq:OU}, respectively.
\end{observation}
\begin{proof}
    We denote $z^*$ as the optimal objective value of problem~\eqref{eq:OU}. Let $\underline{z}$ denote the optimal objective value of problem~\eqref{eq:OU} where $\Gamma=0$. Similarly, $\bar{z}$ denotes the optimal objective value of problem~\eqref{eq:OU} where $\Gamma \geq n$. We have $\underline{z}=\min_{j\in[n]} \bar{c}_j$ and $\bar{z}=\min_{j\in[n]} \bar{c}_j+\hat{c}_j$. Furthermore, $\underline{z} \leq z^* \leq \bar{z}$.
\end{proof}

\begin{observation}
\label{obs: psi(I,Gammaone)}
    For any $I \in \I$ and any $\Gammaone\in \setGammaI$, the adversary attacks the first $\Gammaone$ items in $I$ to maximize $\psi(I,\Gammaone)$.
\end{observation}
\begin{proof}
    If $\Gammaone = \Gamma$, the cheapest item for the second stage is the one that takes the lower value between $\min_{j \in I : \gamma_j=1} \bar{c}_j+\hat{c}_j$ and $\min_{j\in\set{i\in I}{\gamma_i=0}\cup \bI} \bar{c}_j$. Observe that as $\bar c_j\leq \bar c_j +\hat c_j$ for each $j\in[n]$, the first term can be replaced by $\min_{j\in[n]} \bar{c}_j+\hat{c}_j$, leading to $\psi(I,\Gammaone)= \min \{\min_{j\in[n]} \bar{c}_j+\hat{c}_j, \min_{j\in \bI}\bar{c}_j, \min_{j\in I: \gamma_j=0} \bar{c}_j \}$.
    Therefore, to maximize $\psi(I,\Gammaone)$, the adversary attacks the items with smallest $\bar{c}_j$ in $I$. By Assumption~\ref{assump: item sorted by bar c}, the adversary attacks the first $\Gammaone$ items in $I$.
    If $\Gammaone < \Gamma$, the adversary has enough budget to attack any item $i \in \bI$ with $x_i=1$. Thus, the cheapest item is the one that takes the minimum between $\min_{j\in\set{i\in I}{\gamma_i=1}\cup \bI} \bar{c}_j+\hat{c}_j$ and $\min_{j \in I: \gamma_j=0}\bar{c}_j $.
    Similarly, by replacing the first term as before, we have $\psi(I,\Gammaone)= \min \{\min_{j\in[n]} \bar{c}_j+\hat{c}_j,  \min_{j \in I: \gamma_j=0}\bar{c}_j \}$. 
    Therefore, to maximize $\psi(I,\Gammaone)$, the adversary again targets the items with smallest $\bar{c}$ in $I$, specifically the first $\Gammaone$ items in $I$.
\end{proof}

\subsection{Expression of $\Psi(I)$}
\label{subsec: Expression of Psi(I)}
In this subsection, we provide closed-form expressions of the value of $\Psi(I)$ for a query set $I\in \I$ where $|I|<\Gamma$,  $|I|=\Gamma$ and $|I|>\Gamma$, respectively. Consequently, given a query set $I$, $\Psi(I)$ can be computed in polynomial time.

\begin{lemma}
\label{obs: |I| < Gamma}
     Let $I\in \I$ such that $|I| < \Gamma$, then $\Psi(I)=\min_{j \in [n]} \bar{c}_j+\hat{c}_j$.
\end{lemma}

\begin{proof}
     If $|I| < \Gamma$, then $\Gammazero \ge 1$.
     This implies that in the worst case,  the adversary attacks all items in $I$ and also targets any item in $\bI$ that could potentially be selected in the second stage. Specifically, $\delta_j=1$ whenever $x_j=1$. In other words, we do not have the opportunity to select an item whose $\delta$ is definitely equal to 0. Hence, the optimal value is $\min_{j \in [n]} \bar{c}_j + \hat{c}_j$.
\end{proof}

\begin{lemma}
\label{lemma: |I| = Gamma}
    Let $I = \{j_1, j_2, \dots, j_{\Gamma}\}$, then 
    \begin{equation*}
    \Psi(I)=\min \left\{ \min_{j\in[n]} \bar{c}_j+\hat{c}_j, \max \{ \min_{j \in \bI} \bar{c}_j, \bar{c}_{j_{\Gamma}} \}   \right\}.
    \end{equation*}
\end{lemma}

\begin{proof}
    In order to determine the value of $\Psi(I)=\max_{\Gammaone\in[\Gamma]_0}\psi(I,\Gammaone)$, we consider two cases: $\Gammaone=\Gamma$ and $\Gammaone<\Gamma$.
    If $\Gammaone=\Gamma$, then $\psi(I,\Gammaone) = \min \{\min_{l\in[\Gamma]} \bar{c}_{j_l}+ \hat{c}_{j_l}, \min_{j \in \bI}\bar{c}_j\} = \min \{\min_{j\in[n]} \bar{c}_j+ \hat{c}_j,\min_{j \in \bI}\bar{c}_j\}$. The latter equivalence holds because 
    $\bar{c}_j \leq \bar{c}_j + \hat{c}_j$ for each $j \in [n]$.
    If $\Gammaone<\Gamma$, by Observation~\ref{obs: psi(I,Gammaone)}, we have that $\psi(I,\Gammaone) = \min\{\min_{j \in [n]\setminus \{j_{\Gammaone+1},\dots,j_{\Gamma} \}} \bar{c}_j+\hat{c}_j, \bar{c}_{j_{\Gammaone+1}}  \} = \min\{\min_{j \in [n]} \bar{c}_j+\hat{c}_j, \bar{c}_{j_{\Gammaone+1}}  \}$, so $\psi(I,\Gammaone)$ is maximized when $\Gammaone=\Gamma-1$, in which case $\psi(I,\Gamma-1)= \min \{ \min_{j\in[n]} \bar{c}_j+\hat{c}_j, \bar{c}_{j_{\Gamma}}\}$. 
    Thus, 
    \begin{align*}
    \Psi(I) &= \max \left(\min \left\{\min_{j\in[n]} \bar{c}_j+ \hat{c}_j,\min_{j \in \bI}\bar{c}_j\right\}, \min \left\{ \min_{j\in[n]} \bar{c}_j+\hat{c}_j, \bar{c}_{j_{\Gamma}}\right \} \right )\\
    &= \min  \left\{ \min_{j\in[n]} \bar{c}_j+\hat{c}_j, \max \left\{\min_{j \in \bI} \bar{c}_j,  \bar{c}_{j_{\Gamma}} \right\}   \right \}.
    \end{align*}
\end{proof}

\begin{lemma}
\label{lemma: |I| > Gamma}
    Let $I = \{j_1,j_2,\dots,j_{\Gamma},j_{\Gamma+1},\dots,j_m \}$ with $\Gamma+1 \leq m \leq n-1$, then 
    \begin{equation*}
    \Psi(I)=\min \left\{ \min_{j\in[n]} \bar{c}_j+\hat{c}_j, \max \left\{ \min \left\{ \bar{c}_{j_{\Gamma+1}}, \min_{j \in \bI}\bar{c}_j\right\}, \bar{c}_{j_{\Gamma}} \right\} \right\}.
    \end{equation*}
\end{lemma}
\begin{proof}
    As before, we consider two cases: $\Gammaone=\Gamma$ and $\Gammaone<\Gamma$. If $\Gammaone=\Gamma$, then $\psi(I,\Gammaone) = \min \{\min_{l\in[\Gamma]} \bar{c}_{j_l}+ \hat{c}_{j_l}, \bar{c}_{j_{\Gamma+1}}, \min_{j \in \bI}\bar{c}_j\} =  \min \{\min_{j\in[n]} \bar{c}_j+ \hat{c}_j, \bar{c}_{j_{\Gamma+1}}, \min_{j \in \bI}\bar{c}_j\}$. If $\Gammaone<\Gamma$, by Observation~\ref{obs: psi(I,Gammaone)}, we have $\psi(I,\Gammaone) = \min\{\min_{j \in [n]\setminus \{j_{\Gammaone+1},\dots,j_m \}} \bar{c}_j+\hat{c}_j, \bar{c}_{j_{\Gammaone+1}}  \} = \min\{\min_{j \in [n]} \bar{c}_j+\hat{c}_j, \bar{c}_{j_{\Gammaone+1}}  \}$, so $\psi(I,\Gammaone)$ is maximized when $\Gammaone=\Gamma-1$, in which case $\psi(I,\Gamma-1)= \min \{ \min_{j\in[n]} \bar{c}_j+\hat{c}_j, \bar{c}_{j_{\Gamma}}\}$. Thus, 
    \begin{align*}
    \Psi(I) &= \max \left(\min \left\{\min_{j\in[n]} \bar{c}_j+ \hat{c}_j, \bar{c}_{j_{\Gamma+1}}, \min_{j \in \bI}\bar{c}_j\right\}, \min \left\{ \min_{j\in[n]} \bar{c}_j+\hat{c}_j, \bar{c}_{j_{\Gamma}}\right\}
    \right)\\
    &=\min \left\{ \min_{j\in[n]} \bar{c}_j+\hat{c}_j, \max \left\{ \min \left\{ \bar{c}_{j_{\Gamma+1}}, \min_{j \in \bI}\bar{c}_j \right\}, \bar{c}_{j_{\Gamma}} \right\} \right\}.
    \end{align*}
\end{proof}

\subsection{NP-hardness of the general case}
\label{subsec: NP-hardness of the general case}


We now show that \eqref{eq:OU} is a hard problem if we do not restrict the set of possible queries $\I$. In the next sections, we discuss special cases for $\I$ that result in polynomially solvable cases. In particular, problem~\eqref{eq:OU} is not necessarily hard if optimizing a linear function over $\I$ is a hard problem, as the knapsack case demonstrates.

\begin{proposition}\label{prop1}
    Problem~\eqref{eq:OU} is NP-hard in the strong sense.
\end{proposition}

\begin{proof}
    Consider the independent set decision problem: Given a graph $G(V,E)$ and a positive integer $K \leq |V|$, does $G$ contain an independent set of size $K$ or more, i.e. a subset $V' \subseteq V$ such that $|V'| \geq K$ and any two vertices in $V'$ are not adjacent? We construct an instance of problem~\eqref{eq:OU} by setting  $n=|V|$, $\Gamma=K$, $\bar{c}_j = 0$ and $\hat{c}_j = 1$, $\forall j \in [n]$, and letting $\I$ consist of the independent sets of the graph $G(V,E)$.
    From Lemma~\ref{obs: |I| < Gamma}, we know that, if $|I| < \Gamma$, the optimal objective value of this instance is $1$. If $|I| \geq \Gamma$, after the discovery of information, there always exists at least one item $j$ with $\delta_j=0$, either in set $I$, or in set $\bI$, so the optimal objective value is $0$.
    Therefore, the independent set decision problem has a \emph{yes} solution if and only if the optimal objective value of this instance is $0$. 
\end{proof}


\subsection{Polynomial-time algorithms for special cases}
\label{subsec: Polynomial-time algorithms for special cases}
We have proved that problem~\eqref{eq:OU} is NP-hard in general. However, for some special cases of $\I$, the problem can be solved in polynomial time. In this section, we provide polynomial algorithms for the two cases $\I^{SEL}$ (selection query sets) and $\I^{KP}$ (knapsack query sets) running in $O(n)$ and $O(n\log n)$, respectively.

\subsubsection{Selection query sets}

In this subsection, we present a polynomial solution for problem~\eqref{eq:OU} with $\I^{SEL}=\set{I \subseteq [n]}{|I| \leq q}$. Notice that while Assumption~\ref{assump: item sorted by bar c} used to prove the correctness of the algorithm, the latter does actually not need to sort the items.

\begin{proposition}
    If $\I = \I^{SEL}$ and $\Gamma \leq q$,
    problem~\eqref{eq:OU} can be solved in $O(n)$ by querying in the optimal $I^*$ the $q$ items with smallest $\bar{\pmb{c}}$, leading to a solution cost of $\min\{ \min_{j \in [\Gamma]}\bar{c}_j + \hat{c}_j, \bar{c}_{\Gamma+1}\}$.
\end{proposition}

\begin{proof} 
    First, we determine the value of $\Psi(I^*)=\max_{\Gammaone\in[\Gamma]_0}\psi(I^*,\Gammaone)$. The worst-case occurs when the adversary uses his entire budget to attack the items with the smallest nominal costs, which are in $I^*$. Thus, we have $\Psi(I^*)= \min\{ \min_{j \in [\Gamma]}\bar{c}_j + \hat{c}_j, \bar{c}_{\Gamma+1}\}$.

    Now, considering any solution $I \in \I$, we have 
    \begin{align*}
    \Psi(I)= \max_{\pmb{\gamma} \in \UG} \min_{j\in[n]} \max_{\pmb{\delta} \in \UG(I,\pmb{\gamma})} \bar{c}_j+\delta_j\hat{c}_j
    & \geq  \max_{\pmb{\gamma} \in \UG}\max_{\pmb{\delta} \in \UG(I,\pmb{\gamma})} \min_{j\in[n]}  \bar{c}_j+\delta_j\hat{c}_j \\
    & = \max_{\pmb{\delta} \in \UG}\min_{j\in[n]}  \bar{c}_j+\delta_j\hat{c}_j \\
    & \geq \min_{j\in[n]}  \bar{c}_j+\delta'_j\hat{c}_j 
    \end{align*}
where $\delta'_j = 1$, for each $j \in [\Gamma]$ and $\delta'_j = 0$, for each $j \in \{\Gamma+1,\Gamma+2,\dots,n \}$. Therefore, $\Psi(I) \geq \min\{ \min_{j \in [\Gamma]}\bar{c}_j + \hat{c}_j, \bar{c}_{\Gamma+1}\}=\Psi(I^*)$.
Finally, the complexity of the algorithm follows from Theorem~\ref{thm:On}.
\end{proof}
If $\Gamma > q$, then we fall into the case described in Lemma~\ref{obs: |I| < Gamma}.


\subsubsection{Knapsack query sets}

 Let us recall that $\I^{KP}=\set{I \subseteq[n]}{\sum_{i\in I} a_i \leq C}$. In this subsection, we discuss how to construct an optimal solution $I^* \in \I^{KP}$ in polynomial time. We propose in Lemmas~\ref{lemma: Psi(I) leq Psi(I')} and~\ref{lemma: Psi(I)=Psi(I')} two rules for identifying dominant solutions, which are also valid for the general case. Finally, we present a pseudo-code for solving the whole problem with $\I^{KP}$ in $O(n \log n)$.

\begin{lemma}
\label{lemma: Psi(I) leq Psi(I')}
    Let $I=\{j_1, j_2, \dots, j_{\Gamma} \}$ and $I' = \{j'_1, j'_2, \dots, j'_{\Gamma}\}$ with $j'_{\Gamma} > j_{\Gamma}$. Then, $\Psi(I) \leq \Psi(I')$.
\end{lemma}
\begin{proof}
    Let us introduce the additional notation $\bI= \{i_1, i_2, \dots, i_{n-\Gamma} \}$ and $\bI' =  \{i'_1, i'_2,\dots, i'_{n-\Gamma} \}$.
    Since $|I| = |I'| = \Gamma$, from Lemma~\ref{lemma: |I| = Gamma}, we know that $\Psi(I)= \min \{ \min_{j\in[n]} \bar{c}_j+\hat{c}_j, \max \{\bar{c}_{i_1},  \bar{c}_{j_{\Gamma}} \}    \}$ and $\Psi(I') = \min \{ \min_{j\in[n]} \bar{c}_j+\hat{c}_j, \max \{\bar{c}_{i'_1},  \bar{c}_{j'_{\Gamma}} \}    \}$. Furthermore, we have $j'_{\Gamma} > j_{\Gamma} \geq \Gamma$. Thus, $j'_{\Gamma} \geq \Gamma+1$ and $i'_1 < j'_{\Gamma}$. Therefore, $\Psi(I') = \min \{ \min_{j\in[n]} \bar{c}_j+\hat{c}_j, \bar{c}_{j'_{\Gamma}}    \}$.

    To compare $\Psi(I)$ and $\Psi(I')$, we discuss the following cases:
    \begin{itemize}
        \item If $i_1 > j_{\Gamma}$, then $\Psi(I)= \min \{ \min_{j\in[n]} \bar{c}_j+\hat{c}_j, \bar{c}_{i_1}\}$. Since $i_1 = \Gamma +1$ and $j'_{\Gamma} \geq \Gamma+1$, we have $j'_{\Gamma} \geq i_1$. Thus $\Psi(I) \leq \Psi(I')$.
        \item If $i_1 < j_{\Gamma}$, then $\Psi(I)= \min \{ \min_{j\in[n]} \bar{c}_j+\hat{c}_j, \bar{c}_{j_{\Gamma}}\}$. Since $j'_{\Gamma} > j_{\Gamma}$, $\Psi(I) \leq \Psi(I')$.
    \end{itemize}

    Overall, we have $\Psi(I) \leq \Psi(I')$.
\end{proof}

\begin{lemma}
\label{lemma: Psi(I)=Psi(I')}
    Let $I = \{j_1, j_2, \dots, j_{\Gamma}\}$ and $I' = I \cup \{j_{\Gamma+1}, j_{\Gamma+2},\dots,j_m\}$ with $\Gamma+1 \leq m \leq n-1$. Then, $\Psi(I)=\Psi(I')$.
\end{lemma}
\begin{proof}
    Let us introduce the additional notation $\bI= \{i_1, i_2, \dots, i_{n-\Gamma} \}$ and $\bI' = \{i'_1,i'_2,\dots,i'_{n-m} \}$.
    Since $|I| = \Gamma$ and $|I'| > \Gamma$, from Lemmas~\ref{lemma: |I| = Gamma} and~\ref{lemma: |I| > Gamma}, we know that $\Psi(I)= \min \{ \min_{j\in[n]} \bar{c}_j+\hat{c}_j, \max \{\bar{c}_{i_1},  \bar{c}_{j_{\Gamma}} \}    \}$ and $\Psi(I')=\min \{ \min_{j\in[n]} \bar{c}_j+\hat{c}_j, \max \{ \min \{ \bar{c}_{j_{\Gamma+1}}, \bar{c}_{i'_1}\}, \bar{c}_{j_{\Gamma}} \} \}$.

    To compare $\Psi(I)$ and $\Psi(I')$, we discuss the following cases:
    \begin{itemize}
        \item If $j_{\Gamma} < j_{\Gamma+1} = i_1 < i'_1$, then $j_{\Gamma}=\Gamma$, $j_{\Gamma+1} = i_1 = \Gamma+1$, $i'_1\geq \Gamma+2$ and $\Psi(I)= \min \{ \min_{j\in[n]} \bar{c}_j+\hat{c}_j, \bar{c}_{\Gamma+1} \} = \Psi(I')$.
        \item If $j_{\Gamma} < i_1 =i'_1 < j_{\Gamma+1}$, then $j_{\Gamma}=\Gamma$, $i_1=i'_1 = \Gamma+1$, $j_{\Gamma+1} \geq \Gamma+2$ and $\Psi(I)=\min \{ \min_{j\in[n]} \bar{c}_j+\hat{c}_j, \bar{c}_{\Gamma+1} \} =\Psi(I')$.
        \item If $i_1 = i'_1 < j_{\Gamma} < j_{\Gamma+1}$, then $\Psi(I)=\min \{ \min_{j\in[n]} \bar{c}_j+\hat{c}_j, \bar{c}_{j_{\Gamma}} \} = \Psi(I')$.
    \end{itemize}

    Overall, we have $\Psi(I) = \Psi(I')$.
\end{proof}




Lemma~\ref{lemma: Psi(I)=Psi(I')} implies that, if we have a solution $I$ such that $|I|=\Gamma$, then it is meaningless to observe any additional items with $\bar{c}_j$ greater than or equal to $\max_{j \in I} \bar{c}_j$ (i.e. $c_{j_{\Gamma}}$). Furthermore, from Lemma~\ref{lemma: Psi(I) leq Psi(I')}, we know that, an optimal solution $I$ with $|I|= \Gamma$ seeks to minimize the value of $\max_{j \in I} \bar{c}_j$. However, such a solution must also satisfy $\sum_{i\in I} a_i \leq C$, leading to solving
\begin{equation}
\label{eq: opt I}
    \min_{I=\{j_1,\ldots,j_\Gamma\} \subseteq [n]} \set{j_{\Gamma}}{\sum_{i\in I} a_i \leq C}.
\end{equation}
If problem~\eqref{eq: opt I} is infeasible, then by Lemma~\ref{obs: |I| < Gamma}, the optimal solution of problem~\eqref{eq:OU} selects the item with $\min_{j \in [n]} \bar{c}_j+\hat{c}_j$. In other words, if no solution $I \in \I$ satisfies $|I| \geq \Gamma$, then we can simply select the item with $\min_{j \in [n]} \bar{c}_j+\hat{c}_j$, making further information discovery unnecessary.

We propose a polynomial algorithm to solve problem~\eqref{eq: opt I}. For each $j \in \{\Gamma, \Gamma+1,\dots, n\}$, the solution is constructed as $I(j)=\{j\}\cup A(j)$ where the set $A(j)$ contains $\Gamma-1$ items with the smallest $\pmb{a}$ values from the set of indices $[j-1]$. If this solution satisfies the constraint $\sum_{i\in I(j)} a_i \leq C$, $I(j)$ is an optimal solution. If not, we increment $j$ by $1$. If, until $j=n$, we have not found a solution $I \in \I^{KP}$ such that $|I|=\Gamma$, then problem~\eqref{eq: opt I} is infeasible. At each iteration, we can obtain the set $A(j)$ by re-sorting the items according to $\pmb{a}$. Thus, the whole problem can be solved in $O(n^2 \log n)$.

The previous algorithm can be improved by observing that $A(j)$ can be constructed based on $A(j-1)$ by finding the item $i$ in $\argmin_{i \in [j] \setminus A(j-1) } a_i$ in $O(n)$. Based on this idea, we propose a more effective algorithm (see Algorithm~\ref{algo:1-selection}). Specifically, we first place all items in a list $L$ in non-decreasing order of $\bar{\pmb{c}}$ and place the first $\Gamma$ items in the list $I$ as an initial solution. At each iteration, we add an additional item in $I$, in sequence according to the order. We next remove the item with the highest value of $\pmb{a}$ in $I$. Once the constraint $\sum_{i\in I} a_i \leq C$ is satisfied, the algorithm returns the optimal objective value: $\min \left\{ \min_{j\in[n]} \bar{c}_j+\hat{c}_j, \max \{ \min_{j \in \bI} \bar{c}_j,  \max_{j \in I}\bar{c}_j \}   \right\}$, as given in Lemma~\ref{lemma: |I| = Gamma}. If, by the end, we have not found a set $I$ containing $\Gamma$ items and satisfying the constraint, the algorithm returns the optimal objective value: $\min_{j\in[n]} \bar{c}_j+\hat{c}_j$, as given in Lemma~\ref{obs: |I| < Gamma}. Hence, problem~\eqref{eq:OU} with $\I=\I^{KP}$ can be solved in $O(n^2)$, which can be reduced to $O(n \log n)$ using a binary heap.

We have thus proved the following result.

\begin{proposition}\label{prop3}
If $\I=\I^{KP}$, problem~\eqref{eq:OU} can be solved in $O(n\log n)$ using Algorithm~\ref{algo:1-selection}.
\end{proposition}

\begin{algorithm}[H]
\caption{1-Selection with DDID}
\label{algo:1-selection}
    $I\leftarrow\emptyset$, $\bI\leftarrow\emptyset$\;
    Compute upper bound $\bar{z}= \min_{j\in[n]} \bar{c}_j+\hat{c}_j$\; 
    Sort all items based on $\bar{\pmb{c}}$ non-decreasing in $L$\; 
    Add the first $\Gamma$ items of $L$ to $I$\;
    $k \leftarrow \Gamma+1$\;
    \While{$k \leq n$}{
        Add $k$-th item of $L$ to $I$\;
        Identify the item $j$ with the highest value of $a_j$ in $I$, then remove it from $I$ and place it into $\bI$\; 
        
        \lIf{$\sum_{i\in I} a_i \leq C$}{
            \Return $\min \left\{\bar{z} , \max \{ \min_{i \in \bI} \bar{c}_i,  \max_{j \in I}\bar{c}_j \}   \right\}$
        }
        \Else{
            \lIf{$k < n$}{
                $k \leftarrow k+1$
            }
            \lElse{
                \Return $\bar{z}$
            }}   
    }
\end{algorithm}

\section{Constraint uncertainty}
\label{sec:CU}

\label{DDID with uncertainty constraint}
In this section, we study problem~\eqref{eq:CU}. We derive basic properties in Section~\ref{subsec: Additional definitions and a basic property} and show that the general case is NP-hard in Section~\ref{subsec: con NP-hardness of the general case}. In Section~\ref{subsec: MILP reformulation}, we propose a linear programming formulation for computing the cost of $I$, and dualize it to provide an MILP reformulation of the problem. Last,
we prove in Section~\ref{subsec:ISEL} that the special case of $\I^{SEL}$ can be solved in polynomial time.

\subsection{Additional definitions and basic properties}
\label{subsec: Additional definitions and a basic property}

To simplify expressions, for any $I\in \I$ and $\Gammaone\in \setGammaI$, we define the cost of solution $I$ using exactly $\Gammaone$ deviations on $I$ and $\Gammazero=\Gamma-\Gammaone$ deviations on $\bI$ as
\begin{equation*}
\phi(I,\Gammaone) = \max_{\pmb{\gamma} \in \Delta^{\Gammaone}_=} \min_{\pmb{x}\in \X}\set{\sum_{i\in[n]} c_i x_i}{\sum_{i\in[n]} \delta_i x_i \leq b\ \forall \pmb{\delta}\in \Delta^{\Gamma-\Gammaone}_=(I,\pmb{\gamma})}.
\end{equation*}
Observe that $\Phi(I) = \max_{\Gammaone\in\setGammaI} \phi(I,\Gammaone)$.
Throughout the section, we assume the following without loss of generality.
\begin{assumption}
\label{assump:cnondec}
    All items are sorted so that the costs $\pmb{c}$ are non-decreasing.
\end{assumption}

If $p \leq b$ or $\Gamma \leq b$, the problem always has a feasible and optimal solution with the costs of the first $p$ items. So, in this section, we will only discuss the cases where $p > b$ and $\Gamma > b$. Furthermore, let us denote by $Q=\max_{I\in\I}|I|$ the cardinality of the largest set in $\I$. We provide next conditions that $Q$ must satisfy to guarantee that problem~\eqref{eq:CU} is feasible. While this condition is interesting on its own, it will be leveraged later in Section~\ref{subsec:ISEL} to prove the polynomiality of the problem in the case of $\I^{SEL}$.
\begin{proposition}
\label{prop:conditions}
If $n \geq p + Q$, $p > b$ and $\Gamma > b$, problem~\eqref{eq:CU} is feasible if and only if
\begin{enumerate}
\item $\Gamma \leq b+Q$,
\item $p \leq 2b+Q-\Gamma+1$.
\end{enumerate}
\end{proposition}
\begin{proof}
    Sufficiency($\Leftarrow$): Suppose that $\Gamma \leq b+Q$ and $p \leq 2b+Q-\Gamma+1$, we show that there is a feasible solution for $\pmb{x}\in \X$. 
    \begin{itemize}
        \item If $\Gammaone=Q$, then $\sum_{i\in \bI}\delta_i = \Gamma - \Gammaone \leq b$. Because $n \geq p + Q$, there is a feasible solution $\pmb{x}$ such that $\sum_{i\in \bI}x_i = p$.
        \item If $\Gammaone<Q$, we discuss two cases. If $\sum_{i\in \bI}\delta_i = \Gamma - \Gammaone \leq b$, then $n \geq p + Q$ implies that there is a feasible solution $\pmb{x}$ such that $\sum_{i\in \bI}x_i = p$. Otherwise, $\sum_{i\in \bI}\delta_i = \Gamma - \Gammaone > b$, and there is a feasible solution $\pmb{x}$ if and only if $p \leq Q-\Gammaone+b$. The worst-case of the latter condition arises when $\Gammaone$ takes its largest value (not greater than $Q$). We claim that this when $\Gammaone=\Gamma-b-1$: since $b+1 \leq \Gamma \leq b+Q$, we also have $0 \leq \Gamma-b-1 \leq Q-1$, proving the claim. Since $p \leq 2b+Q-\Gamma+1$, the problem is feasible.
    \end{itemize}
    Necessity($\Rightarrow$): To see that the first condition is necessary, it is enough to observe that if $\Gamma > b+Q$, the adversary has enough budget to target all observed items as well as $b+1$ non-observed items, preventing any solution $\pmb{x} \in \X$ to satisfy $\sum_{i\in [n]} \delta_ix_i \leq b$.
 
 Let us now consider the second condition. Suppose that $\Gamma \leq b+Q$ and $p > 2b+Q-\Gamma+1$. If $\Gammaone=\Gamma-b-1$, then $\sum_{i \in \bI} \delta_i = \Gamma-\Gammaone=b+1$, so there is a feasible solution $\pmb{x}$ if and only if $p \leq Q-\Gammaone+b$. Since $p > 2b+Q-\Gamma+1$, the problem is not feasible.
\end{proof}

\subsection{NP-hardness of the general case}
\label{subsec: con NP-hardness of the general case}

We now show that problem~\eqref{eq:CU} is hard to solve. In contrast to Propositions~\ref{prop1} and \ref{prop3}.

\begin{proposition}
\label{prop:CUhardKP}
Problem~\eqref{eq:CU} is NP-hard even when $\I=\I^{KP}$.
\end{proposition}
\begin{proof}
We consider a reduction from the partition problem, which given the integers $k_1,\ldots,k_m$ that sum up to $2K$, seeks a set $S$ of cardinality $|S|=m/2$ such that $\sum_{i\in S}k_i=\sum_{i\in[m]\setminus S}k_i=K$. We define the following instance of~\eqref{eq:CU} by setting $n=m+1$, $a_i=k_i$ and $c_i=-k_i$ for each $i\in[m]$, $a_{m+1}=K+1$, $c_{m+1}=-K-1$, $C=K$, $p=m/2$, $\Gamma=1$ and $b=0$. We prove next that there exists a \emph{yes} solution to the partition problem if and only if the optimal solution cost of the constructed instance of~\eqref{eq:CU} has a cost of $-K$.

Let $S\subseteq[m]$ be some set and define $I=S$. We see that $I$ is feasible for the outermost minimization problem if and only if $k(S)\leq C$. 

Suppose first that $|S|<m/2$ and consider the solution $\pmb{\gamma}=(0,\ldots,0)$ for the maximization problem. Then, for any solution $\pmb{x}\in \X$, there must be $i'\in[n]\setminus S$ such that $x_{i'}=1$, so $\pmb{x}$ turns out to be infeasible for the constraint associated with $\pmb{\delta}\in\UG(I,(0,\ldots,0))$ such that $\delta_{i'}=1$.

Therefore, we assume next that $|S|=m/2$ and we look at the maximization problem.  Suppose $\gamma_{i'}=1$ for some $i'\in S$. In this case, there is a feasible solution for the inner minimization problem for which $x_{i}=1$ for $i\in S\setminus \{i'\}$ and $x_{m+1}=1$, with cost less than $-K$. Otherwise, suppose $\pmb{\gamma}=(0,\ldots,0)$. Then, the only feasible solution for the inner minimization problem is $x_i=1$ if and only if $i\in S$, with cost exactly $-K$. Thus, this latter choice of $\pmb{\gamma}$ is optimal for the maximization problem and the overall cost of the solution is $-k(S)$, which is equal to $-K$ if and only if $S$ is a \emph{yes} solution to the partition problem.
\end{proof}

\subsection{Linear programming formulations}
\label{subsec: MILP reformulation}
\subsubsection{Linear formulation for $\Phi$}

This section aims to propose a linear programming formulation for $\Phi(I)$. Let us start by rewriting function $\Phi(I)$ as
\begin{equation*}
    \Phi(I) = \max_{\pmb{\gamma} \in \UG} \min_{\pmb{x}\in \X}\set{\pmb{c}^T\pmb{x}}{\sum_{i \in I} \gamma_i x_i + \sum_{i \in \bI} \delta_i x_i  \leq b,\quad \forall \pmb{\delta}\in \UG(I,\pmb{\gamma})},
\end{equation*}
where, as before, $\bI$ is the complement of $I$. Observe the minimization problem is a robust combinatorial optimization problem with one robust constraint with budgeted uncertainty (and budget equal to $\Gamma - \sum_{i\in I} \gamma_i$). Lemma 2 from \cite{Alvarez-MirandaLT13} provides a method to compute the robust counterpart of this problem. In order to apply this result to our minimization problem more intuitively, we express the uncertain intervals of $\pmb{\gamma}$ and $\pmb{\delta}$ as $\gamma_i \in [\gamma_i, \gamma_i]$, for each $i \in I$, and $\delta_i\in[0,1]$, for each $i \in \bI$, where we can relax the binary restriction on $\pmb{\delta}$ because the budget ($\Gamma - \sum_{i\in I} \gamma_i$) is integer. Therefore, we obtain:

\begin{align*}
\Phi(I) = \max_{\pmb{\gamma} \in \UG} \min\Bigg[
&\min \set{\pmb{c}^T\pmb{x}}{\sum_{i\in I} \gamma_i x_i
+\sum_{i\in \bI} x_i \leq b, \sum_{i \in [n]} x_i \geq p, \pmb{x} \in \{0,1\}^{n}},\\
&\left.\min \set{\pmb{c}^T\pmb{x}}{\sum_{i\in I} \gamma_i x_i
+\Gamma - \sum_{i\in I} \gamma_i\leq b,\sum_{i \in [n]} x_i \geq p, \pmb{x} \in \{0,1\}^{n}}\right].
\end{align*}
We next introduce a property of the two inner minimization problems. Given a binary vector $\pmb{a}'$ and an integer $b'$, we define the following general set
\begin{equation*}
\X'=\set{\pmb{x}\in\mathbb{R}^n_+}{(\pmb{a}')^T \pmb{x} \leq b', \sum_{i \in [n]} x_i \geq p, x_i \leq 1, \forall i \in [n]}.
\end{equation*}

\begin{observation}
\label{obs:unimodular}
    $\X'$ is an integral polytope.
\end{observation}

\begin{proof}
    We prove that the matrix defining the constraints is Totally Unimodular (TU). For the first two constraints: $(\pmb{a}')^T \pmb{x} \leq b'$ and $ \sum_{i \in [n]} x_i \geq p$, every entry in this coefficient matrix with only two rows on $\pmb{x}$ is either 0 or 1. So, each square non-singular submatrix of this matrix is unimodular. Therefore, the first two rows form a totally unimodular. Last, the constraints $\pmb{x}\leq \pmb{1}$ form an identity matrix, so overall, the constraints of the problem form a TU matrix.
\end{proof}

Thanks to Observation~\ref{obs:unimodular}, we can relax the integrality restriction on $\pmb{x}$. Introducing an epigraphic formulation for the outermost minimum and dualizing the two inner linear programs, we obtain
\begin{align*}
\Phi(I) = \max \;&\omega\\
\mbox{s.t.}\;& \omega \leq p\cdot\alpha^0-\one^T\pmb{\beta}^0-b\cdot\lambda^0\\
& \alpha^0-\beta_i^0-\GammaI \cdot \lambda^0\leq c_i, \quad & \forall i\in I \\
 &\alpha^0-\beta_i^0-\lambda^0\leq c_i, \quad & \forall i\in \bI \\
 & \omega \leq p\cdot\alpha^1-\one^T\pmb{\beta}^1-(b+\sum_{i\in I}\GammaI - \Gamma)\lambda^1\\
& \alpha^1-\beta_i^1-\GammaI \cdot \lambda^1\leq c_i, \quad &\forall i\in I \\
 &\alpha^1-\beta_i^1\leq c_i, \quad &\forall i\in \bI \\
 & \sum_{i\in I}\GammaI \leq \Gamma \\
 & \pmb{\gamma}\in\{0,1\}^{|I|},\;\alpha^0,\pmb{\beta}^0,\lambda^0,\alpha^1,\pmb{\beta}^1,\lambda^1 \geq 0.
\end{align*}

For any $\GammaI \in \setGammaI$, let us introduce here the notation $I_\GammaI\subseteq I$ to denote the set of the first $\GammaI$ elements of $I$.

\begin{proposition}
\label{pro:smallestcost}
For any $\GammaI \in \setGammaI$, there exists a solution $\pmb{\gamma}^*$ to the maximization problem defining $\Phi(I)$ such that $\gamma^*_i=1$ for $i\in I_\GammaI$.
\end{proposition}

\begin{proof}
Suppose that there are $i,j\in I$, with $c_i\leq c_j$, such that $\GammaI=0$ and $\gamma_j=1$. We prove next that swapping the indices $i$ and $j$ in vector $\pmb{\gamma}$ leads to a solution whose cost is not less than the cost of the original solution. To simplify the notations throughout, we define $\tilde{c}_i^{\ell} = -c_i + \alpha^\ell$ for $\ell\in\{0,1\}$.
Observe that in an optimal solution, $\beta_k^0=\pospart{\tilde{c}_k^0-\gamma_k\lambda^0}$ for each $k\in I$ and  $\beta_k^0=\pospart{\tilde{c}_k^0-\lambda^0 }$ for each $k\in\bI$, $\beta_k^1=\pospart{\tilde{c}_k^1-\gamma_k\lambda^1 }$ for each $k\in I$ and  $\beta_k^1=\pospart{\tilde{c}_k^1}$ for each $k\in\bI$.
This implies that $\beta_i^0=\pospart{\tilde{c}_i^0}$, and  $\beta_j^0=\pospart{\tilde{c}_j^0-\lambda^0 }$. Furthermore, in an optimal solution,
\begin{equation*}
\omega = \min\left\{p\cdot\alpha^0-\one^T\pmb{\beta}^0-b\cdot\lambda^0,
p\cdot\alpha^1-\one^T\pmb{\beta}^1-(b+\sum_{i\in I}\GammaI - \Gamma)\lambda^1\right\}.
\end{equation*}

Next, we define the swapped vector $\pmb{\gamma}'$ with $\gamma'_i=1$, $\gamma'_j=0$ and $\gamma'_k=\gamma_k$ for $k\in I\setminus\{i,j\}$. 
We further define $\beta_i^{'0}= \pospart{ \tilde{c}_i^0- \lambda^0}$, $\beta_j^{'0} = \pospart{\tilde{c}_j^0}$ and $\beta_k^{'0}=\beta_k^{0}$, for $k\in I\setminus\{i,j\}$. $\omega '$ is defined as follows,
\begin{equation*}
\omega' = \min\left\{p\cdot\alpha^0-\one^T\pmb{\beta}^{'0}-b\cdot\lambda^0,
p\cdot\alpha^1-\one^T\pmb{\beta}^{'1}-(b+\sum_{i\in I}\GammaI - \Gamma)\lambda^1\right\}.
\end{equation*}
To simplify the notation, we introduce that $\omega_1 = p\cdot\alpha^0-\one^T\pmb{\beta}^{0}-b\cdot\lambda^0$, $\omega_2 = p\cdot\alpha^1-\one^T\pmb{\beta}^1-(b+\sum_{i\in I}\GammaI - \Gamma)\lambda^1$, $\omega '_1 = p\cdot\alpha^0-\one^T\pmb{\beta}^{'0}-b\cdot\lambda^0$ and $\omega '_2 = p\cdot\alpha^1-\one^T\pmb{\beta}^{'1}-(b+\sum_{i\in I}\GammaI - \Gamma)\lambda^1$.
First, to prove that $\omega '_1 \geq \omega_1$, we consider the following cases, given $\alpha^0, \lambda^0 $ fixed:
\begin{itemize}
\item If $\tilde{c}_i^0-\lambda^0 \geq 0 $ and $\tilde{c}_j^0-\lambda^0 \geq 0$: $\omega '_1 - \omega_1 = 0$.
\item If $\tilde{c}_i^0-\lambda^0 \geq 0$ and $0 \leq \tilde{c}_j^0 \leq \lambda^0$: $\omega'_1 - \omega_1 = -\tilde{c}_j^0 + \lambda^0 \geq 0$.
\item If $0 \leq \tilde{c}_i^0 \leq \lambda^0$ and $0 \leq \tilde{c}_j^0 \leq \lambda^0$: $\omega'_1 - \omega_1 = \tilde{c}_i^0 - \tilde{c}_j^0 \geq 0$.
\item If $\tilde{c}_i^0 \geq 0$ and $\tilde{c}_j^0 \leq 0$: $\omega'_1 - \omega_1 = \tilde{c}_i^0 - \pospart{\tilde{c}_i^0 - \lambda^0} \geq 0$.
\item If $\tilde{c}_i^0 \leq 0$ and $\omega'_1 - \omega_1 = 0$.
\end{itemize}
In all cases, $\omega'_1 \geq \omega_1$, similarly, with regard to $\beta^1$, the same conclusion holds true: $\omega'_2 \geq \omega_2$. Therefore, $\omega ' \geq \omega$, indicating that $\gamma'$ represents a solution at least as good as $\gamma$, proving the result.
\end{proof}

Thanks to Proposition~\ref{pro:smallestcost}, given a budget $\Gammaone$ of deviations for $\pmb{\gamma}$, the adversary should attack those having the smallest $c_i$. 
Thus, $\Phi(I)$ can be reformulated as
the smallest $\eta$ such that for all $\Gammaone \in\setGammaI$ it holds that:
\begin{align*}
  \eta \geq \max\; & \omega\\
 \mbox{s.t.}\; & \omega \leq p\cdot\alpha^0-\one^T\pmb{\beta}^0-b\cdot\lambda^0 & &&[e^0]\\
& \alpha^0-\beta_i^0-\lambda^0\leq c_i,  & \forall i\in I_\GammaI &&[f_i^0]\\
& \alpha^0-\beta_i^0\leq c_i,  & \forall i\in I\setminus I_\GammaI &&[g_i^0]\\
& \alpha^0-\beta_i^0-\lambda^0\leq c_i,  & \forall i\in \bI &&[h_i^0]\\
 & \omega \leq p\cdot\alpha^1-\one^T\pmb{\beta}^1-(b+\Gammaone - \Gamma)\lambda^1 & &&[e^1]\\
& \alpha^1-\beta_i^1-\lambda^1\leq c_i,  & \forall i\in I_\GammaI &&[f_i^1]\\
& \alpha^1-\beta_i^1\leq c_i,  & \forall i\in I\setminus I_\GammaI &&[g_i^1]\\
 & \alpha^1-\beta_i^1\leq c_i,  & \forall i\in \bI &&[h_i^1]\\
 &  \alpha^0,\pmb{\beta}^0,\lambda^0,\alpha^1,\pmb{\beta}^1,\lambda^1 \geq 0
\end{align*}
We introduce the dual variables $e^0,f_i^0,g_i^0,h_i^0,e^1,f_i^1,g_i^1,h_i^1$ corresponding to each constraint of the above reformulation, as denoted by the corresponding brackets. Specifically, $f_i^0$ and $f_i^1$ are defined for all $i\in I_\GammaI$; $g_i^0$ and $g_i^1$ are defined for all $i\in I\setminus I_\GammaI$; and $h_i^0$ and $h_i^1$ are defined for all $i\in \bI$. Dualizing the above linear programs, we obtain that $\Phi(I)$ is the smallest $\eta$ such that for all $\Gammaone\in\setGammaI$ it holds that:
\begin{align}
\eta \geq \min\; & \sum_{i \in I_\GammaI} c_i (f_i^0+f_i^1) + \sum_{i\in I\setminus I_\GammaI} c_i (g_i^0+g_i^1) + \sum_{i\in \bI}c_i (h_i^0+h_i^1)\label{eq:dual:obj}\\
\mbox{s.t.}\; & e^0 + e^1 = 1 \label{eq:dual:first}\\
    &-p\cdot e^0 + \sum_{i \in I_\GammaI} f_i^0 + \sum_{i\in I\setminus I_\GammaI}g_i^0 + \sum_{i\in \bI} h_i^0 \geq 0 \\
    & e^0-f_i^0 \geq 0, & \forall i \in I_\GammaI\\
    & e^0-g_i^0 \geq 0, & \forall i\in I\setminus I_\GammaI \\
    & e^0-h_i^0 \geq 0, & \forall i\in \bI \\
    & b\cdot e^0 - \sum_{i \in I_\GammaI} f_i^0 - \sum_{i\in \bI} h_i^0 \geq 0 \\
    & -p\cdot e^1 + \sum_{i \in I_\GammaI} f_i^1 + \sum_{i\in I\setminus I_\GammaI}g_i^1 + \sum_{i\in \bI} h_i^1 \geq 0 \\
    & e^1-f_i^1 \geq 0, & \forall i \in I_\GammaI\\
    & e^1-g_i^1 \geq 0, & \forall i\in I\setminus I_\GammaI \\
    & e^1-h_i^1 \geq 0, & \forall i\in \bI \\
    & (b+\Gammaone-\Gamma)e^1 - \sum_{i \in I_\GammaI} f_i^1  \geq 0 \\
    &e^0,e^1 \geq 0\\
    &f_i^0, f_i^1 \geq 0,&\forall i \in I_\GammaI\\
    &g_i^0, g_i^1 \geq 0, &\forall i\in I\setminus I_\GammaI\\
    &h_i^0, h_i^1 \geq 0, & \forall i\in \bI\label{eq:dual:last}
\end{align}

Adding the index $\GammaI$ to the variables of~\eqref{eq:dual:obj}--\eqref{eq:dual:last} corresponding to $\GammaI$, we obtain that
\begin{align}
     \Phi(I) =
    \min \;&\eta\\
    \mbox{s.t.}\; & \eta \geq \min\Biggl\{\sum_{i \in I_\GammaI} c_i (f_{i,\Gammaone}^0+f_{i,\Gammaone}^1) + \sum_{i\in I\setminus I_\GammaI} c_i (g_{i,\Gammaone}^0+g_{i,\Gammaone}^1) + \sum_{i\in \bI}c_i (h_{i,\Gammaone}^0+h_{i,\Gammaone}^1)\;\Biggr|\label{eq:explicitmin}\\
    & (e_\GammaI^0,\bbf_\GammaI^0,\bg_\GammaI^0,\bh_\GammaI^0,e_\GammaI^1,\bbf_\GammaI^1,\bg_\GammaI^1,\bh_\GammaI^1)\mbox{ satisfies }\eqref{eq:dual:first}-\eqref{eq:dual:last}\Biggr\},\quad
    \forall \Gammaone\in\setGammaI.\nonumber
\end{align}
We finally obtain the linear programming formulation for $\Phi(I)$ by observing that the minimization can be removed in the right-hand-side of~\eqref{eq:explicitmin}:
\begin{align}
     \Phi(I) =
    \min \;&\eta\\
    \mbox{s.t.}\;& \eta \geq \sum_{i \in I_\GammaI} c_i (f_{i,\Gammaone}^0+f_{i,\Gammaone}^1) + \sum_{i\in I\setminus I_\GammaI} c_i (g_{i,\Gammaone}^0+g_{i,\Gammaone}^1) + \sum_{i\in \bI}c_i (h_{i,\Gammaone}^0+h_{i,\Gammaone}^1),\hspace{-5.5cm}\nonumber \\
    & & \forall \Gammaone\in\setGammaI \label{constr-phi-start}\\
    & e_{\Gammaone}^0 + e_{\Gammaone}^1 = 1,& \forall \Gammaone\in\setGammaI \\
    & -p\cdot e_{\Gammaone}^0 + \sum_{i \in I_\GammaI} f_{i,\Gammaone}^0 + \sum_{i\in I\setminus I_\GammaI}g_{i,\Gammaone}^0 + \sum_{i\in \bI} h_{i,\Gammaone}^0 \geq 0, \hspace{-5.5cm} & \forall \Gammaone\in\setGammaI\\
    & e_{\Gammaone}^0-f_{i,\Gammaone}^0 \geq 0, & \forall i \in I_\GammaI, \Gammaone\in\setGammaI\\
    & e_{\Gammaone}^0-g_{i,\Gammaone}^0 \geq 0, & \forall i\in I\setminus I_\GammaI, \Gammaone\in\setGammaI\\
    & e_{\Gammaone}^0-h_{i,\Gammaone}^0 \geq 0, & \forall i\in \bI, \Gammaone\in\setGammaI \\
    & b\cdot e_{\Gammaone}^0 - \sum_{i \in I_\GammaI} f_{i,\Gammaone}^0 - \sum_{i\in \bI} h_{i,\Gammaone}^0 \geq 0,  & \forall \Gammaone\in\setGammaI\\
    & -p\cdot e_{\Gammaone}^1 + \sum_{i \in I_\GammaI} f_{i,\Gammaone}^1 + \sum_{i\in I\setminus I_\GammaI}g_{i,\Gammaone}^1 + \sum_{i\in \bI} h_{i,\Gammaone}^1 \geq 0,  \hspace{-5.5cm} & \forall \Gammaone\in\setGammaI\\
    & e_{\Gammaone}^1-f_{i,\Gammaone}^1 \geq 0, & \forall i \in I_\GammaI, \Gammaone\in\setGammaI\\
    & e_{\Gammaone}^1-g_{i,\Gammaone}^1 \geq 0, & \forall i\in I\setminus I_\GammaI, \Gammaone\in\setGammaI\\
    & e_{\Gammaone}^1-h_{i,\Gammaone}^1 \geq 0, & \forall i\in \bI, \Gammaone\in\setGammaI \\
    & (b+\Gammaone-\Gamma)e_{\Gammaone}^1 - \sum_{i \in I_\GammaI} f_{i,\Gammaone}^1  \geq 0, & \forall \Gammaone\in\setGammaI \label{constr-phi-fin}\\
    & e_{\Gammaone}^0, e_{\Gammaone}^1 \geq 0, &\forall \Gammaone\in\setGammaI\\
    & f_{i,\Gammaone}^0, f_{i,\Gammaone}^1 \geq 0, &\forall i \in I_\GammaI,\forall \Gammaone\in\setGammaI \label{constr_f}\\
    & g_{i,\Gammaone}^0, g_{i,\Gammaone}^1 \geq 0, &\forall i\in I\setminus I_\GammaI, \forall \Gammaone\in\setGammaI \label{constr_g}\\
    &h_{i,\Gammaone}^0, h_{i,\Gammaone}^1 \geq 0, &\forall i\in \bI, \forall \Gammaone\in\setGammaI \label{constr_h}
\end{align}

\subsubsection{Mixed-integer programming formulation for problem~\eqref{eq:CU}}
\label{sec:mip}

To present the mixed-integer linear programming reformulation of this subsection, we first need to define the set following set of binary vectors, which is bijection with $\I$:
\begin{equation*}
\W = \set{\pmb{w}\in\{0,1\}^n}{\exists I \in \I \text{  such that } w_i=1 \Longleftrightarrow i \in I}.
\end{equation*}
We furthermore assume that we know a compact formulation for $\W$ in the following sense.
\begin{assumption}
There is a matrix $A$ and a vector $\pmb{b}$ whose dimensions are polynomial in $n$ and such that $\W=\set{\pmb{w}\in\{0,1\}^n}{A\pmb{w} \leq \pmb{b}}$.
\end{assumption}
We introduce next the binary optimization variables $w_i$, $\forall i \in [n]$, and the variables $u_{i,\Gammaone}$, $\forall i \in [n], \Gammaone \in [\Gamma]_{0}$. These variables have the following meanings: $w_i=1$ if $i \in I$, and $w_i=0$ otherwise. Further,  given a $\Gammaone$, $u_{i,\Gammaone}=1$ if $i\in I_\GammaI$, and $u_{i,\Gammaone}=0$ otherwise. Observe that $\Gammaone$ cannot exceed the number of the observed coefficients.
Therefore, we introduce the additional binary variables $z_{\GammaI}$, $\forall \Gammaone \in [\Gamma]_0$, where $z_{\Gammaone}=1$ if $\Gammaone \leq \sum_{i\in [n]} w_i$ and $z_{\Gammaone}=0$ otherwise. Therefore, we obtain the following MILP reformation for the problem $\min_{I \in \I} \Phi(I)$, containing polynomially many variables and constraints and big-$M$ coefficients:
\begin{align}
\min \;&\eta \label{mip:start}\\
\mbox{s.t.}\;&\eta \geq \sum_{i \in [n]}c_i (f_{i,\Gammaone}^0+f_{i,\Gammaone}^1) + \sum_{i\in [n]} c_i (g_{i,\Gammaone}^0+g_{i,\Gammaone}^1) + \sum_{i\in [n]}c_i (h_{i,\Gammaone}^0+h_{i,\Gammaone}^1), \hspace{-8cm} \nonumber\\
& & \forall \Gammaone \in [\Gamma]_{0} \label{constr1-start}\\
&e_{\Gammaone}^0 + e_{\Gammaone}^1 = 1,& \forall \Gammaone\in[\Gamma]_0  \\
    &-p\cdot e_{\Gammaone}^0 + \sum_{i \in [n]} f_{i,\Gammaone}^0 + \sum_{i\in [n]}g_{i,\Gammaone}^0 + \sum_{i\in [n]} h_{i,\Gammaone}^0 \geq 0,\hspace{-8cm} & \forall \Gammaone \in [\Gamma]_{0} \\
    & e_{\Gammaone}^0-f_{i,\Gammaone}^0 \geq u_{i,\Gammaone} -1, & \forall i \in [n], \Gammaone \in [\Gamma]_{0}\\
    &e_{\Gammaone}^0-g_{i,\Gammaone}^0 \geq w_i - u_{i,\Gammaone} -1, & \forall i\in [n], \Gammaone \in [\Gamma]_{0} \\
    & e_{\Gammaone}^0-h_{i,\Gammaone}^0 \geq -w_i, & \forall i\in [n], \Gammaone \in [\Gamma]_{0}  \\
    &b\cdot e_{\Gammaone}^0 - \sum_{i \in [n]} f_{i,\Gammaone}^0 - \sum_{i\in [n]} h_{i,\Gammaone}^0 \geq 0,& \forall \Gammaone \in [\Gamma]_{0} \\
    &-p\cdot e_{\Gammaone}^1 + \sum_{i \in [n]} f_{i,\Gammaone}^1 + \sum_{i\in [n]}g_{i,\Gammaone}^1 + \sum_{i\in [n]} h_{i,\Gammaone}^1 \geq 0, \hspace{-8cm} & \forall \Gammaone \in [\Gamma]_{0} \\
    & e_{\Gammaone}^1-f_{i,\Gammaone}^1 \geq u_{i,\Gammaone} -1, & \forall i \in [n], \Gammaone \in [\Gamma]_{0}\\
    & e_{\Gammaone}^1-g_{i,\Gammaone}^1 \geq w_i - u_{i,\Gammaone} -1, & \forall i\in [n], \Gammaone \in [\Gamma]_{0} \\
    & e_{\Gammaone}^1-h_{i,\Gammaone}^1 \geq - w_i, & \forall i\in [n], \Gammaone \in [\Gamma]_{0} \\
    & (b+\Gammaone-\Gamma)e_{\Gammaone}^1 - \sum_{i \in [n]} f_{i,\Gammaone}^1  \geq 0,& \forall \Gammaone \in [\Gamma]_{0} \label{constr1-fin}\\
    & u_{i,\Gammaone} \leq w_i, & \forall i \in [n] , \Gammaone \in [\Gamma]_{0} \label{constr2-start}\\
    & u_{i,\Gammaone} \geq u_{j,\Gammaone} - 1 + w_i, & \forall i \in [n], j \in [n]: i < j , \Gammaone \in [\Gamma]_{0}\label{constr2-fin}\\
    & f_{i,\Gammaone}^0 \leq u_{i,\Gammaone}, & \forall i \in [n], \Gammaone \in [\Gamma]_{0}\label{constr3-start}\\
    & f_{i,\Gammaone}^1 \leq u_{i,\Gammaone}, & \forall i \in [n], \Gammaone \in [\Gamma]_{0}\\
    & g_{i,\Gammaone}^0 \leq w_i-u_{i,\Gammaone}, & \forall i \in [n], \Gammaone \in [\Gamma]_{0}\\
    & g_{i,\Gammaone}^1 \leq w_i-u_{i,\Gammaone}, & \forall i \in [n], \Gammaone \in [\Gamma]_{0}\\
    & h_{i,\Gammaone}^0 \leq 1-w_i, & \forall i \in [n], \Gammaone \in [\Gamma]_{0} \\
    & h_{i,\Gammaone}^0 \leq 1-w_i, & \forall i \in [n], \Gammaone \in [\Gamma]_{0} \label{constr3-fin}\\
    & M \cdot z_{\Gammaone} \geq \sum_{i\in [n]} w_i - \Gammaone,& \forall \Gammaone \in [\Gamma]_{0} \label{constr4-start}\\
    & M \cdot (z_{\Gammaone}-1) \leq \sum_{i\in [n]} w_i - \Gammaone,& \forall \Gammaone \in [\Gamma]_{0}\\
    &\sum_{i \in [n]} u_{i,\Gammaone} \geq \Gammaone - M\cdot(1-z_{\Gammaone}),& \forall \Gammaone \in [\Gamma]_{0}\\
    &\sum_{i \in [n]} u_{i,\Gammaone} \geq \sum_{i\in [n]} w_i - M\cdot z_{\Gammaone},& \forall \Gammaone \in [\Gamma]_{0}\\
    &\sum_{i \in [n]} u_{i,\Gammaone} \leq \Gammaone,& \forall \Gammaone \in [\Gamma]_{0}\\
    &\sum_{i \in [n]} u_{i,\Gammaone} \leq \sum_{i\in [n]} w_i,& \forall \Gammaone \in [\Gamma]_{0} \label{constr4-fin}\\
    & \pmb{z} \in \{0,1\}^{\Gamma+1}\\
    & \pmb{w} \in \W \\
    & \pmb{u} \in \{0,1\}^{n \times (\Gamma+1 )} \\
    & e_{\Gammaone}^0,f_{i,\Gammaone}^0,g_{i,\Gammaone}^0,h_{i,\Gammaone}^0,e_{\Gammaone}^1,f_{i,\Gammaone}^1,g_{i,\Gammaone}^1,h_{i,\Gammaone}^1 \geq 0\hspace{-8cm} & \forall i \in [n], \Gammaone \in [\Gamma]_{0} \label{mip:end}
\end{align}
In the above MILP formulation, Constraints~(\ref{constr1-start}--\ref{constr1-fin}) are reformulated from Constraints~(\ref{constr-phi-start}--\ref{constr-phi-fin}) using the variables $\pmb{w}$ and $\pmb{u}$. Constraints~(\ref{constr2-start}--\ref{constr2-fin}) impose that if $u_j=1$ for $j>i$ and $w_i=1$, then $u_i$ must also be set to one, in accordance with the definition of $I_\GammaI$.
Notice from (\ref{constr_f}--\ref{constr_h}) that the variables $\pmb{f}$, $\pmb{g}$, $\pmb{h}$ are related to the items in the sets $I_\GammaI$, $I\setminus I_\GammaI$, $\bI$, respectively. Therefore, these variables should be zero for any items not included in their respective sets. This requirement is enforced by Constraints~(\ref{constr3-start}--\ref{constr3-fin}). Finally, Constraints~(\ref{constr4-start}--\eqref{constr4-fin}) ensure that for each $\Gammaone \in [\Gamma]_0$, $\sum_{i \in [n]} u_{i,\Gammaone} = \min\{\Gammaone, \sum_{i\in [n]}w_i\}$.
Obviously, big-$M$ can be set to $n$.

\subsection{Selection query sets}
\label{subsec:ISEL}

We prove in this subsection that in the special case $\I^{SEL}=\set{I\subseteq[n]}{|I| \leq q}$ (so $Q=q$), the optimal solution always picks specific $q$ items from the ranked list where the items are sorted by $c_i$. Hence, thanks to Theorem~\ref{thm:On}, the resulting optimization problem can be solved in $O(n)$. Notice that here also, Assumption~\ref{assump:cnondec} is only made to ease the proofs, as the algorithm itself does not require to sort the items. 

\begin{theorem}
\label{thm:optimal}
If $n\geq p + q$, $p > b$ and $\Gamma > b$, there is an optimal solution to problem~\eqref{eq:CU} with $\I=\I^{SEL}$ such that $I^*=\{b+1,b+2,\ldots,b+q\}$, whose value is $\sum_{i=1}^{b} c_i + \sum_{i=\Gamma+1}^{p-b+\Gamma} c_i$.
\end{theorem}
\begin{proof}
We assume that the two conditions of Proposition~\ref{prop:conditions} are verified. Our aim is to prove that the optimal solution is reached at $I^*=\{b+1,b+2,\ldots,b+q\}$. First, in order to determine the value of $\Phi(I^*)=\max_{\Gammaonestar\in[\min(q,\Gamma )]_0}\phi(I,\Gammaonestar)$, we consider two cases based on the values of $\Gammaonestar=\sum_{i\in I^*}\GammaI$ and $\Gammazerostar=\Gamma-\Gammaonestar$: $\Gammazerostar$ $\geq b$ and $\Gammazerostar < b$.

If $\Gammazerostar \geq b$, we can pick at most $b$ items from the set $\bI^*\cup \set{i\in I^*}{\GammaI = 1}$, the remaining items must be picked from the set $\set{i\in I^*}{\GammaI = 0}$. By Proposition~\ref{pro:smallestcost}, the adversary always targets the first $\Gammaonestar$ items in the set $I^*$. Therefore, the set $\set{i\in I^*}{\GammaI = 1}$ includes the indices $\{b+1,b+2,\dots, b+\Gammaonestar\}$. Consequently, the set $\bI^*\cup \set{i\in I^*}{\GammaI = 1}$ covers the indices $\{1,2,\dots,b+\Gammaonestar,b+q+1,b+q+2,\dots,n\}$, while the set 
$\set{i\in I^*}{\GammaI = 0}$ consists of the indices $\{b+\Gammaonestar+1, b+\Gammaonestar+2,\dots, b+q\}$.
Thus, we have
\begin{equation*}
\phi(I^{*},\Gammaonestar)=\sum_{i=1}^{b} c_i + \sum_{i=b+\Gammaonestar+1}^{b+\Gammaonestar+(p-b)}c_i = \sum_{i=1}^{b} c_i + \sum_{i=b+\Gammaonestar+1}^{p+\Gammaonestar}c_i.
\end{equation*}

If $\Gammazerostar < b$, let us construct an optimal selection by picking the cheapest items possible. We start by picking the first $b$ items, which are in $\bI^*$. As $\Gammazerostar$ of these items shall be affected by the uncertainty, we can next pick at most $b-\Gammazerostar$ items from the set $\set{i\in I^*}{\gamma=1}$, corresponding to the indices $\{b+1,b+2,\dots,b+b-\Gammazerostar\}$ by Proposition~\ref{pro:smallestcost}. This implies that the first $2b - \Gammazerostar$ ($=b+(b-\Gammazerostar)$) items can be picked. The remaining $p-(2b - \Gammazerostar)$ items must be picked from the set $\bI^*\cup \set{i\in I^*}{\GammaI = 0}$, corresponding the indices $\{b+\Gammaonestar+1, b+\Gammaonestar+2, \dots, n\}$. Depending on the value of $p$, we discuss two subcases: $p \leq 2b -\Gammazerostar$ and $p > 2b - \Gammazerostar$. If $p \leq 2b - \Gammazerostar$, then $\phi(I^{*},\Gammaonestar) = \sum_{i=1}^{p} c_i$. Otherwise,
\begin{equation*}
\phi(I^{*},\Gammaonestar) = \sum_{i=1}^{2b - \Gammazerostar} c_i + \sum_{i=b+\Gammaonestar+1}^{b+\Gammaonestar+(p-(2b - \Gammazerostar))} c_i = \sum_{i=1}^{2b - \Gammazerostar} c_i + \sum_{i=b+\Gammaonestar+1}^{p-b+\Gamma} c_i.
\end{equation*}

Overall, we have
\begin{equation*}
\phi(I^{*},\Gammaonestar)=\begin{cases}
    \sum_{i=1}^{p} c_i & \text{if } \Gammazerostar \leq 2b-p\\
    \sum_{i=1}^{2b - \Gammazerostar} c_i + \sum_{i=b+\Gammaonestar+1}^{p-b+\Gamma} c_i & \text{if } 2b - p<\Gammazerostar < b\\
    \sum_{i=1}^{b} c_i + \sum_{i=b+\Gammaonestar+1}^{p+\Gammaonestar}c_i & \text{if } \Gammazerostar \geq b,
\end{cases}
\end{equation*}
where we observe that $2b-p< b$ because of the assumption $p>b$. 

To simplify the discussion of the function $\phi(I^*,\Gammaonestar)$ below, we rewrite it as follows by substituting $\Gammazerostar$ with $\Gamma-\Gammaonestar$.
\begin{equation*}
\phi(I^{*},\Gammaonestar)=\begin{cases}
    \sum_{i=1}^{b} c_i + \sum_{i=b+\Gammaonestar+1}^{p+\Gammaonestar}c_i & \text{if } \Gammaonestar \leq \Gamma-b\\
    \sum_{i=1}^{2b - \Gamma+\Gammaonestar} c_i + \sum_{i=b+\Gammaonestar+1}^{p-b+\Gamma} c_i & \text{if } \Gamma-b<\Gammaonestar < \Gamma-2b+p\\
    \sum_{i=1}^{p} c_i & \text{if } \Gammaonestar \geq \Gamma-2b+p
\end{cases}
\end{equation*}
We next show that 
\begin{equation}
\label{eq:Gammaonemax}
\Gamma - b\in\argmax_{\Gammaonestar\in[\min(q,\Gamma )]_0}\phi(I^*,\Gammaonestar).
\end{equation}
In the case where $\Gammaonestar \leq \Gamma-b$, the function $\phi(I^{*},\Gammaonestar)$ is non-decreasing in $\Gammaonestar$, so it is maximized when $\Gammaonestar=\Gamma-b$, in which case $\phi(I^{*},\Gamma-b)= \sum_{i=1}^{b} c_i + \sum_{i=\Gamma+1}^{p-b+\Gamma} c_i$. In the case where $\Gamma-b<\Gammaonestar < \Gamma-2b+p$, the function $\phi(I^{*},\Gammaonestar)$ is non-increasing in $\Gammaonestar$, so it is maximized when $\Gammaonestar=\Gamma-b+1$. Furthermore,
\begin{equation*}
\phi(I^{*},\Gamma-b+1)=\sum_{i=1}^{b+1} c_i + \sum_{i=\Gamma+2}^{p-b+\Gamma} c_i \leq \sum_{i=1}^{b} c_i + \sum_{i=\Gamma+1}^{p-b+\Gamma} c_i =\phi(I^{*},\Gamma-b).
\end{equation*}
Last, in the case where $\Gammaonestar \geq \Gamma-2b+p$, the function is independent of $\Gammaonestar$ and $\sum_{i=1}^{p} c_i\leq \phi(I^{*},\Gamma-b)$.
To conclude the proof of~\eqref{eq:Gammaonemax}, it remains to show that there always exists a $\Gammaonestar\in[\min(q,\Gamma )]_0$ such that $\Gammaonestar=\Gamma-b$: since $b+1\leq \Gamma \leq b+q$, we know that $1\leq \Gamma-b \leq q$, which proves this fact.

Now, we consider any solution $I\in\I^{SEL}$. Consider the indicator function $\indicator_{\pmb{\delta}^T\pmb{x}> b}$. Observe that because $\pmb{x}$ is binary, we can substitute the constraint on $\pmb{x}$ with a penalization term in the objective function, so for $M$ large enough, we have
\begin{align*}
    \Phi(I) 
    = 
    \max_{\gamma \in \UG} \min_{\pmb{x}\in \X}\max_{\pmb{\delta}\in \UG(I,\pmb{\gamma})}\pmb{c}^T\pmb{x}+M\indicator_{\pmb{\delta}^T\pmb{x}> b}
    &\geq 
    \max_{\pmb{\gamma} \in \UG} \max_{\pmb{\delta}\in \UG(I,\pmb{\gamma})} \min_{\pmb{x}\in \X}\pmb{c}^T\pmb{x}+M\indicator_{\pmb{\delta}^T\pmb{x}> b}\\
    &=\max_{\pmb{\delta} \in \UG} \min_{\pmb{x}\in \X}\pmb{c}^T\pmb{x}+M\indicator_{\pmb{\delta}^T\pmb{x}> b}   \\
    &\geq \min_{\pmb{x}\in \X}\pmb{c}^T\pmb{x}+M\indicator_{(\pmb{\delta}')^T\pmb{x}> b},    \label{eq:minmax_ineq}
\end{align*}
where $\delta'_i = 1$, for each $i \in [\Gamma]$ and $\delta'_i = 0$, for each $i \in \{\Gamma+1,\Gamma+2,\dots,n \}$. We see that the optimal solution to the minimization problem picks at most $b$ items from the first $\Gamma$ items, and the remaining items must be picked from the set $ \{\Gamma+1,\Gamma+2,\dots,n \}$. Therefore,  $\Phi(I) \geq \sum_{i=1}^{b} c_i + \sum_{i=\Gamma+1}^{p-b+\Gamma} c_i=\Phi(I^*)$.
\end{proof}

If $\I = \I^{SEL}$, thanks to Theorem~\ref{thm:optimal}, the problem can be solved in $O(n)$ by relying on Theorem~\ref{thm:On}. Notice that while the optimal value of this case does not depend on the value of $q$, the value of $q$ determines the feasibility of the problem. We state this result in the following corollary.

\begin{corollary}
If $\I=\I^{SEL}$, problem~\eqref{eq:CU} can be solved in $O(n)$.
\end{corollary}

\section{Numerical experiments}
\label{sec:num}

We provide next a short numerical illustration of the solving capability of the MILP reformulation derived in Section~\ref{sec:mip} for the case $\I=\I^{KP}$.
We generated 200 instances with $n\in\{20,40,60,80,100\}$, $p\in\{n/5,n/10\}$ and $\Gamma\in\{n/5,n/10\}$. For each parameter choice of $(n,p,\Gamma)$, we randomly generated 10 instances. In each instance, $b$ is a random integer between $1$ and $n$, $\pmb{a}$ and $\pmb{c}$ are two vectors, each containing $n$ random values between $0$ and $50$, and $C$ is a random integer between $1$ and the sum of the values in vector $\pmb{a}$. One of the instances led to an infeasible model. Consequently, the results in Table~\ref{tab:results}, which displays the average solution times across various $n,p,\Gamma$, are obtained from 199 instances. Among these, two instances did not have an optimal solution found within the set time limit of $3600$ seconds, whose values of $n$, $p$ and $\Gamma$ are $\{100,10,20\}$ and $\{100,20,20\}$, respectively. For these two instances, their solving times were recorded as $3600$ seconds, and as a result, the average solution times for these two sizes are marked with an asterisk in the table.
All experiments have been realized in Python 3.6, using Gurobi 11.01, on a processor Intel Xeon E312xx (Sandy Bridge) with a clock speed of 2.29GHz.
\begin{table}[h]
    \centering
    \begin{tabular}{r|r|r|r|r|r|r|r|r}
    $n$ &$p=n/10\quad\Gamma=n/10$ &$p=n/10\quad\Gamma=n/5$ &$p=n/5\quad\Gamma=n/10$ &$p=n/5\quad\Gamma=n/5$\\
    \cline{1-5}
    20	&0.0	&0.1	&0.1	&0.2\\
    40	&0.5	&0.5	&0.5	&3.0\\
    60	&1.0	&1.6	&1.2	&98.3\\
    80	&6.6	&37.0	&13.7	&163.0\\
    100 &6.0	&$414.8^*$	&527.7	&$390.3^*$
    \end{tabular}
    \caption{Average solution times (in seconds) of MILP reformulation}
    \label{tab:results}
\end{table}

Our computational results indicate the viability of formulation~(\ref{mip:start}--\ref{mip:end}), despite its use of big-$M$ constraints and many auxiliary variables. Our compact problem formulation solves all instances with up to 80 items to optimality. The choice of $p$ and $\Gamma$ affects the average computation times, with the smallest choices $p=n/10$ and $\Gamma=n/10$ resulting in the easiest of the considered problems. In particular, the instances we considered with $n=100$, $p=10$ and $\Gamma=10$ can still be solved in 6 seconds on average.

\section{Conclusion}
\label{sec:conc}

Decision-dependent information discovery (DDID) is a powerful and flexible modeling approach that extends classic robust optimization problems by another decision-making stage, in which the decision maker is allowed to query uncertain parameters before choosing a solution. This additional stage means that classic robust optimization approaches cannot be directly applied, which has sparked interest in new solution algorithms for this setting. However, little is known about the complexity of DDID for specific underlying optimization problems.

This paper has started the study of this complexity with particularly simple special cases based on the selection feasibility set. We differentiated between uncertainty in the objective and uncertainty in the constraints. In the first case, we showed that the problem is NP-hard if the set of possible queries is not further restricted, but were able to derive strongly polynomial solution methods in the case of selection query sets or more general knapsack query sets. In the second case, we can give a complete analytical description when the optimization problem is feasible. While the problem is NP-hard even for knapsack query sets, we showed that it remains solvable in strongly polynomial time for selection query sets. Finally, we derived a compact mixed-integer linear programming formulation and showed that using this formulation, problems with knapsack query sets remain solvable within one hour with up to 80 items. Future work will see to extend these results to continuous budget uncertainty set and to problems with objective uncertainty and more than one item to select.

\end{document}